\documentclass[12pt,reqno]{amsart}
\usepackage{geometry}                
\usepackage{amsmath,amssymb}
\usepackage{graphicx}
\usepackage{amssymb,latexsym, amsmath}
\usepackage{amsfonts,amsthm}
\usepackage{amscd}
\usepackage{mathrsfs}
\usepackage{hyperref}
\usepackage{eucal}
\usepackage{epstopdf}
\usepackage{amsgen}
\usepackage{xspace}
\usepackage{verbatim}
\usepackage{stmaryrd}
\usepackage{enumitem}
\newlist{steps}{enumerate}{1}
\setlist[steps, 1]{label = Step \arabic*:}

\newcommand{\gd}{\Delta}

\newcommand{\inpt}[1]{\langle #1 \rangle}

\newcommand{\gw}{\Omega}
\newcommand{\ap}{\alpha}
\newcommand{\ga}{\gamma}
\newcommand{\de}{\delta}
\newcommand{\gb}{\beta}

\newcommand{\gl}{\lambda}

\newcommand{\ms}{\mathscr}

\newcommand{\nb}{\nabla}
\newcommand{\vp}{\varphi}
\newcommand{\ve}{\varepsilon}
\newcommand{\pdr}{\partial}

\newcommand{\beq}{\begin{equation}}
\newcommand{\eeq}{\end{equation}}
\newcommand{\bea}{\begin{align}}
\newcommand{\eea}{\end{align}}
\newcommand{\bthm}{\begin{theorem}}
\newcommand{\ethm}{\end{theorem}}
\newcommand{\bpr}{\begin{proof}}
\newcommand{\epr}{\end{proof}}
\newcommand{\bcl}{\begin{corollary}}
\newcommand{\ecl}{\end{corollary}}
\newcommand{\bpn}{\begin{proposition}}
\newcommand{\epn}{\end{proposition}}
\newcommand{\bre}{\begin{remark}}
\newcommand{\ere}{\end{remark}}
\newcommand{\bdf}{\begin{definition}}
\newcommand{\edf}{\end{definition}}
\newcommand{\bss}{\begin{align*}}
\newcommand{\ess}{\end{align*}}

\newcommand{\bl}{\label}

\newcommand{\mR}{\mathbb{R}}
\newtheorem{theorem}{Theorem}[section]
\newtheorem{corollary}[theorem]{Corollary}

\newtheorem{proposition}[theorem]{Proposition}

\theoremstyle{definition}
\newtheorem{definition}[theorem]{Definition}
\theoremstyle{remark}
\newtheorem{remark}{Remark}

\numberwithin{equation}{section}

\begin{document}

\title[Memristive Hindmarsh-Rose Neural Networks]{Exponential Synchronization of Memristive HIndmarsh-Rose Neural Networks}

\author[Y. You]{Yuncheng You}
\address{Professor Emeritus, University of South Florida, Tampa, FL 33620, USA}
\email{you@mail.usf.edu}
\thanks{}


\date{January 2, 2023}



\begin{abstract} 
A new model of neural networks in terms of the memristive Hindmarsh-Rose equations is proposed. Globally dissipative dynamics is shown with absorbing sets in the state spaces. Through sharp and uniform grouping estimates and by leverage of integral and interpolation inequalities tackling the linear network coupling against the memristive nonlineariry, it is proved that exponential synchronization at a uniform convergence rate occurs when the coupling strengths satisfy the threshold conditions which are quantitatively expressed by the parameters.
\end{abstract}

\maketitle
 
\section{Introduction}

The diffusive Hindmarsh-Rose equations with memristors is a new model for single neuron dynamics proposed by this author very recently \cite{Y}. In this paper, we shall pursue the topic of synchronization of a neural network described by such a model of the memristive and diffusive Hindmarsh-Rose equations with linear couplings. 

Consider a network of $m$ fully coupled neuron cells denoted by $\mathcal{NW} = \{\mathcal{N}_i : i = 1, 2, \cdots, m\}$, where $m \geq 2$ is a positive integer. We shall study the dynamics and the synchronization problem of the following mathematical model of this neural network. Each neuron $\mathcal{N}_i$ in this network is described by the four differential equations
\beq \bl{cHR}
\begin{split}
	\frac{\pdr u_i}{\pdr t} & = \eta_1 \gd u_i +  au_i^2 - bu_i^3 + v_i - w_i + J_e - k_1 \vp (\rho_i) u_i + \sum_{j = 1}^m P (u_j - u_i),  \\
	\frac{\pdr v_i}{\pdr t} & =  \alpha - \beta u_i^2 - v_i,    \\[3pt]
	\frac{\pdr w_i}{\pdr t} & = q (u_i - u_e) - r w_i,   \\[3pt]
	\frac{\pdr \rho_i}{\pdr t} & = \eta_2 \gd \rho_i + u_i - k_2 \rho_i + \sum_{j = 1}^m Q (\rho_j - \rho_i), 
\end{split} 
\eeq
for $1 \leq i \leq m, \; t > 0,\; x \in \gw \subset \mathbb{R}^{n}$ ($n \leq 3$), where $\gw$ is a bounded domain with locally Lipschitz continuous boundary $\partial \gw$. In the membrane potential $u_i$-equations for neuron $\mathcal{N}_i$, the quadratic nonlinear form 
\beq \bl{vp}
	\vp (\rho_i) = c + \ga \rho_i + \delta \rho_i^2, \quad i = 1, 2, \cdots, m.
\eeq
presents the memristive effect, where $\rho_i (t, x)$ stands for the memductance of the memristor and $\vp(\rho_i)$ is the the electromagnetic induction flux with its strength coefficient $k_1$. In this system \eqref{cHR}, the variable $u_i(t,x)$ refers to the membrane electrical potential of a neuron cell, the variable $v_i(t, x)$ called the spiking variable represents the transport rate of the ions of sodium and potassium through the fast ion channels, and the variable $w_i(t, x)$ called the bursting variable represents the transport rate across the neuron membrane through slow channels of calcium and other ions. From biological and mathematical perspectives, the network neuron coupling terms are assumed to be linear in the membrane potential equations and the memristor equations with the coupling strength coefficient $P$ and $Q$ respectively.

All the parameters $a, b, \eta_1, \eta_2, \ap, \gb, q, r,  \de, k_1, k_2, P, Q$ and the external input $J_e$ can be any positive constants, while the reference membrane potential value $u_e$ and the two parameters $c, \ga$ of the memductance \eqref{vp} can be any real number constants. 

We impose the homogeneous Neumann boundary conditions
\begin{equation} \label{nbc}
	\frac{\pdr u_i}{\pdr \nu} (t, x) = 0, \quad \frac{\pdr \rho_i}{\pdr \nu} (t, x) = 0, \quad \text{for} \;\; t > 0,  \; x \in \partial \gw, \quad 1 \leq i \leq m.
\end{equation}
The initial states of the system \eqref{cHR} will be denoted by 
\begin{equation} \bl{inc}
	 u_i^0 (x) = u_i(0, x), \; v_i^0 (x) = v_i(0, x), \; w_i^0 (x) = w_i (0, x), \; \rho_i^0 = \rho_i (0, x), \;\; 1 \leq i \leq m.
\end{equation}

The original Hindmarsh-Rose neuron model \cite{HR} consists of three ordinary differential equations without memristors and it features characterization of the firing-bursting dynamics for neurons, which generates sophisticated Hopf bifurcations and semi-numerical analyses leading to many new solution patterns and collective synchronization behavior \cite{BRS, CS, ET, IG, EI, MFL, CPY, Tr, WS, Su}, especially chaotic bursting dynamics, in comparison with the classical highly nonlinear Hodgkin-Huxley equations \cite{HH} and the two-dimensional FitzHugh-Nagumo equations \cite{FH}.

Global dynamics and synchronization of ensemble neurons modeled by partly diffusive Hindmarsh-Rose equations have been studied by the author's group in recent years \cite{PYS, PY, CPY, PSY}. Such a model of hybrid PDE and ODE reflects the structural feature of neuron cells, which contain the short-branch dendrites receiving incoming signals and the long-branch axon propagating and transmitting outgoing signals through synapses. 

The concept of so-called memristor was coined by Leon Chua \cite{Chua} to denote the effect of electromagnetic flux on moving electric charges. General memristive systems \cite{ChuaK} attracted broad scientific interests in the recent one and half decades since the seminal publication \cite{SS}. Memristors are recognized in many advanced neuron models and applications \cite{Ay, BKG, BB, QW, US2, WP, Wu, XJ} as a different type (other than electrical and chemical) synapsis which can carry dynamically memorized signal information. 

The researches on memristive Hindmarsh-Rose neuron model of ODE have been richly expanding, cf. \cite{Ay, BB, BRS, EE, RJ, RM, SR, Wu} and many references therein. Results on synchronization with the memristive coupling in the ODE models are achieved \cite{Guan, HY, US2, VK, WS, XJ} mainly by the approaches of generalized Hamiltonian functions, Lyapunov exponents, with numerical simulations. However, it is not seen any reported synchronization results on the memristive-based diffusive neural networks modeled by hybrid differential equations. 

In this work we shall rigorously prove a sufficient threshold condition only on the network coupling strengths $P$ and $Q$ to ensure an exponential synchronization of this proposed neural network with a uniform convergence rate, through the approach of dissipative dynamics analysis and the sharp uniform estimates. The methodology in this work can be extended to study more complex neural networks in a broad scope.

\section{\textbf{Formulation and Preliminaries}}

Define two function spaces: 
$$
	E = [L^2 (\gw, \mathbb{R}^4)]^m \quad \text{and}  \quad  \Pi = [H^1 (\gw) \times L^2 (\gw, \mathbb{R}^2) \times H^1 (\gw)]^m
$$ 
which are Hilbert spaces and $H^1 (\gw)$ is a Sobolev space. Also define a Banach space 
$$
	\Psi = [L^2(\gw, \mathbb{R}^3) \times L^4 (\gw)]^m.
$$ 
The norm and inner-product of $L^2(\gw)$ will be denoted by $\| \, \cdot \, \|$ and $\inpt{\,\cdot , \cdot\,}$, respectively. Other space norms will be marked. We use $| \, \cdot \, |$ to denote a vector norm or a set measure in a Euclidean space $\mR^n$. We call $E$ the energy space and $\Psi$ the pivot space. 

The initial-boundary value problem \eqref{cHR}-\eqref{inc} can be formulated into an initial value problem of the evolutionary equation:
\begin{equation} \label{pb}
\begin{split}
	\frac{\partial g}{\partial t} = &\,A g + f (g) + F(g), \;\; t > 0, \\
	&g(0) = g^0 \in E.
\end{split}
\end{equation}
The unknown function in \eqref{pb} is a column vector $g(t) = \text{col}\; (g_1 (t), g_2 (t), \cdots, g_m (t))$, where
$$
	g_i (t) = \text{col}\, (u_i(t, \cdot),\, v_i (t, \cdot ),\, w_i(t, \cdot),\, \rho_i (t, \cdot)), \quad 1 \leq i \leq m,
$$ 
characterizes the dynamics of the neuron $\mathcal{N}_i$. The corresponding initial data function in \eqref{pb} is $g^0 = \text{col}\; (g_1^0, \,g_2^0, \cdots, g_m^0)$, where $g_i^0 = \text{col}\,(u_i^0, \,v_i^0, \,w_i^0, \,\rho_i^0), \, 1 \leq i \leq m$. For notational convenience, we shall use the quasi-norm in the Banach space $\Psi$:
\beq \bl{qnm}
	\interleave g_i(t) \interleave = \|u_i(t)\|^2 + \|v_i(t)\|^2 + \|w_i(t)\|^2 + \|\rho_i(t)\|^4_{L^4}, \;\, 1 \leq i \leq m,
\eeq 
for any 4-dimensional vector function $g_i(t)$ and accordingly $\interleave g(t) \interleave = \sum_{i = 1}^m \interleave g_i(t) \interleave$ for any $4m$-dimensional vector function $g(t) = \text{col}\, (g_1(t), \, g_2(t), \cdots, \,g_m(t))$. 

The closed linear operator $A$ associated with the evolutionary equation in \eqref{pb} and defined by
\begin{equation*} 
	A= \text{diag} \, (A_1, A_2, \cdots, A_m)
\end{equation*}
where
\begin{equation} \label{opA}
A_i =
\begin{pmatrix}
\eta_1 \gd \quad & 0  \quad & 0 & 0.  \\[3pt]
0 \quad & - I \quad  & 0 \quad & 0.  \\[3pt]
0 \quad & 0 \quad & - r I  \quad & 0  \\[3pt]
0 \quad & 0 \quad & 0 \quad  & \eta_2 \gd - k_2 I  
\end{pmatrix}
: \mathcal{D}(A) \rightarrow E, \quad i = 1, 2, \cdots, m,
\end{equation}
with the domain $\mathcal{D}(A) = \{g \in [H^2(\gw) \times L^2 (\gw, \mathbb{R}^2) \times H^2(\gw)]^m: \pdr u_i /\pdr \nu = \pdr \rho_i /\pdr \nu = 0, 1 \leq i \leq m\}$ and the identity operator $I$, is the generator of a $C_0$-semigroup $\{e^{At}\}_{t \geq 0}$ on the space $E$. Since the injection $H^{1}(\gw) \hookrightarrow L^6(\gw)$ is a continuous imbedding for space dimension $n \leq 3$ and by the H\"{o}lder inequality, the nonlinear mapping 
\begin{equation} \label{opf}
f(g) =
\begin{pmatrix}
au_1^2 - bu_1^3 + v_1 - w_1 + J_e - k_1 \vp (\rho_1) u_1 \\[4pt]
\alpha - \beta u_1^2  \\[4pt]
q (u_1 - u_e) \\[4pt]
u_1  \\[4pt]
\cdots \quad \cdots \quad \cdots \\[4pt]
au_m^2 - bu_m^3 + v_m - w_m + J_e - k_1 \vp(\rho_m) u_m  \\[4pt]
\alpha - \beta u_m^2  \\[4pt]
q (u_m - u_e). \\[4pt]
u_m  
\end{pmatrix}
: \Pi \longrightarrow E
\end{equation}
is a locally Lipschitz continuous mapping. The in-network coupling mapping is the vector function
\begin{equation} \label{opg}
F(g) =
\begin{pmatrix}
 \sum_{j = 1}^m P \,(u_j - u_1) \\[3pt]
0  \\[3pt]
0 \\[3pt]
\sum_{j = 1}^m Q (\rho_j - \rho_1)  \\[4pt]
\cdots \quad \cdots \quad \cdots \\[4pt]
 \sum_{j = 1}^m P \,(u_j - u_m) \\[3pt]
0  \\[3pt]
0  \\[3pt] 
\sum_{j = 1}^m Q (\rho_j - \rho_m)
\end{pmatrix}
: E \longrightarrow E .
\end{equation}
We shall consider the weak solutions, cf.\cite[Section XV.3]{CV} of this initial value problem of the evolutionary equation \eqref{pb}.

\begin{definition} \label{D:wksn}
	A $4m$-dimensional vector function $g(t, x)$, where $(t, x) \in [0, \tau] \times \gw$, is called a weak solution to the initial value problem of the evolutionary equation \eqref{pb}, if the following two conditions are satisfied: 
	
	\textup{(i)} $\frac{d}{dt} (g, \zeta) = (Ag, \zeta) + (f(g) + F(g), \zeta)$ is satisfied for a.e. $t \in [0, \tau]$ and any $\zeta \in E^*$;
		
	\textup{(ii)} $g(t, \cdot) \in  C([0, \tau]; E) \cap C^1 ((0, \tau); \Pi)$ and $g(0) = g^0$.
	
\noindent Here $(\cdot , \cdot)$ is the dual product of the dual space $E^*$ versus $E$.
\end{definition}

The following proposition can be proved by the Galerkin approximation method \cite{CV} and the regularity property \cite[Section 4.7]{SY} of the parabolic semigroup $e^{At}$.
\begin{proposition} \label{pps}
	For any given initial state $g^0 \in E$, there exists a unique weak solution $g(t; g^0), \, t \in [0, \tau]$, where $\tau > 0$ may depending on $g^0$, of the initial value problem \eqref{pb}. The weak solution $g(t; g^0)$ satisfies 
\begin{equation} \label{soln}
	g \in C([0, \tau]; E) \cap C^1 ((0, \tau); E) \cap L^2 ((0, \tau); \Pi).
\end{equation}
Moreover, for any initial state $g^0 \in E$, the weak solution $g(t; g^0)$ becomes a strong solution for $t \in (0, \tau)$, which has the regularity
\begin{equation} \bl{ss}
	g \in C((0, \tau]; \Pi) \cap C^1 ((0, \tau); \Pi).
\end{equation}
\end{proposition}

We refer to \cite{CV, SY} for the basics of infinite dimensional dynamical systems or called semiflow if for $t \geq 0$ only. 

\begin{definition} \label{Dabsb}
	Let $\{S(t)\}_{t \geq 0}$ be a semiflow on a Banach space $\ms{X}$. A bounded set $B^*$ in $\ms{X}$ is called an absorbing set of this semiflow if for any given bounded set $B \subset \ms{X}$ there exists a finite time $T_B \geq 0$ depending on $B$, such that $S(t)B \subset B^*$ for all $t  > T_B$. The semiflow is called dissipative if there exists an absorbing set.
\end{definition}

Young's inequality in a general form below will be used throughout. For any two non-negative numbers $x$ and $y$, if $\frac{1}{p} + \frac{1}{q} = 1$ and $p > 1, q > 1$, one has
\beq \bl{Yg}
	x\,y \leq \frac{1}{p} \, \ve x^p + \frac{1}{q} \, C(\ve, p)\, y^q \leq \ve x^p + C(\ve, p)\, y^q, \quad C(\ve, p) = \ve^{-q/p},
\eeq
where constant $\ve > 0$ can be arbitrarily small. Moreover, the Gagliardo-Nirenberg interpolation inequalities \cite[Theorem B.3]{SY} will be used in a crucial step of the proof toward the main result on synchronization of the neural networks in Section 4.

\section{\textbf{Dissipative Dynamics of the Memristive Neural Networks}}

We first prove the global existence of weak solutions in time for the initial value problem \eqref{pb}. Then we show the existence of absorbing set for the solution semiflow in the state spaces $E$ and $\Psi$. Note that $\Pi \subset \Psi \subset E$ are continuous embeddings.

\begin{theorem} \label{T1}
	For any initial state $g^0 \in E$ \textup{(}resp. $g^0 \in \Psi$\textup{)}, there exists a unique global weak solution in the space $E$ \textup{(}resp. $\Psi$\textup{)}, $g(t; g^0) = \textup{col}\, (u_i(t), v_i(t), w_i(t), \rho_i(t): 1 \leq i \leq m), \, t \in [0, \infty)$, to the initial value problem \eqref{pb} of the memristive and diffusive Hindmarsh-Rose equations \eqref{cHR} for the neural network $\mathcal{NW}$. 
\end{theorem}

\begin{proof}
Summing up the $L^2$ inner-products of the $u_i$-equation with $C_1 u_i(t)$ for $1 \leq i \leq m$, where the scaling constant $C_1 > 0$ is to be determined, we get
\begin{equation} \label{u1}
	\begin{split}
	&\frac{C_1}{2} \frac{d}{dt}\, \sum_{i = 1}^m \|u_i (t)\|^2 + C_1 \eta_1 \, \sum_{i = 1}^m \|\nb u_i (t)\|^2  \\
	= &\, \sum_{i = 1}^m \int_\gw C_1 \left[ au_i^3 -bu_i^4  + u_i v_i - u_i w_i + J_e  u_i - k_1 \vp (\rho_i) u_i^2 \right] dx \\
	& - \sum_{i = 1}^m \sum_{j = 1}^m \,\int_\gw C_1 P (u_i - u_j)^2\, dx.
	\end{split} 
\end{equation}
Then sum up the $L^2$ inner-products of the $v_i$-equation with $v_i (t)$ and the $L^2$ inner-products of the $w_i$-equation with $w_i(t)$  for $1 \leq i \leq m$. By Young's inequality \eqref{Yg}, we have
\begin{equation} \label{v1}
	\begin{split}
	&\frac{1}{2} \frac{d}{dt}\, \sum_{i = 1}^m \|v_i (t)\|^2 = \sum_{i = 1}^m \int_\gw \left(\ap v_i - \gb u_i^2 v_i - v_i^2 \right) dx   \\
	\leq\, &\sum_{i = 1}^m \int_\gw \left(\ap v_i +\frac{1}{2} (\gb^2 u_i^4 + v_i^2) - v_i^2 \right) dx    \\
	\leq \,&\sum_{i = 1}^m \int_\gw \left(2\ap^2 + \frac{1}{8} v_i^2 +\frac{1}{2} \gb^2 u_i^4 - \frac{1}{2} v_i^2 \right) dx  \\
	= &\,\sum_{i = 1}^m \int_\gw \left(2\ap^2 +\frac{1}{2} \gb^2 u_i^4 - \frac{3}{8}\, v_i^2 \right) dx
	\end{split}
\end{equation} 
and
\begin{equation} \label{w1}
\begin{split}
	&\frac{1}{2}\, \frac{d}{dt}\, \sum_{i = 1}^m \| w_i (t)\|^2 = \sum_{i = 1}^m \int_\gw \left(q (u_i - u_e) w_i  - r w_i^2 \right) dx  \\
	\leq &\, \sum_{i = 1}^m \int_\gw \left( \frac{q^2}{2r} (u_i - u_e)^2 + \frac{1}{2} r w_i^2 - r w_i^2 \right) dx \leq \sum_{i = 1}^m \int_\gw \left(\frac{q^2}{r} (u_1^2 + u_e^2) - \frac{1}{2}\, r w_i^2 \right) dx.
\end{split} 
\end{equation}
Next sum up the $L^2$ inner-products of the $\rho_i$-equation with $\rho^3_i(t)$ for $1 \leq i \leq m$. We have

\begin{equation*}
\begin{split} 
	&\frac{1}{4}\, \frac{d}{dt}\, \sum_{i = 1}^m \|\rho_i (t)\|^4_{L^4} + 3 \eta_2 \, \sum_{i = 1}^m \|\rho_i(t) \nb \rho_i(t)\|^2  \\
	= &\, \sum_{i = 1}^m \int_\gw \left(u_i \rho^3_i - k_2 \rho_i^4 \right) dx - \sum_{i = 1}^m \sum_{j = 1}^m \,\int_\gw Q (\rho_i - \rho_j)^2 (\rho^2_i + \rho_i \rho_j + \rho^2_j)\, dx     \\ 
	\leq &\,\sum_{i = 1}^m \int_\gw \left(\frac{1}{4 k^3_2} u_i^4 - \left(1 - \frac{3}{4}\right) k_2 \, \rho_i^4)\right) dx = \sum_{i = 1}^m \frac{1}{4}\int_\gw \left(\frac{u_i^4}{k^3_2} - k_2 \rho_i^4 \right) dx,
\end{split}
\end{equation*}
so that
\beq \bl{rh1}
	 \frac{1}{2}\frac{d}{dt}\, \sum_{i = 1}^m \|\rho_i (t)\|^4_{L^4} \leq \sum_{i = 1}^m \int_\gw \frac{1}{2} \left(\frac{u_i^4}{k^3_2} - k_2 \rho_i^4 \right) dx.
\eeq
	
Now choose the constant 
$$
	C_1 = \frac{1}{b} \left(\frac{\gb^2}{2} + \frac{1}{2k^3_2} + 4 \right). 
$$
Then for $1 \leq i \leq m$, we have
\beq \bl{bu4}
	\int_\gw (- \,C_1 b u_i^4)\, dx + \int_\gw \frac{\gb^2}{2} u_i^4\, dx + \int_\gw \frac{1}{2k_2^3} u_i^4 \, dx = \int_\gw (- \,4 u_i^4)\, dx.
\eeq  
The following integral terms on the right-hand side of \eqref{u1} can be estimated:
\beq \bl{nu}
	\int_\gw C_1 au_i^3\, dx \leq \frac{3}{4} \int_\gw u_i^4\, dx + \frac{1}{4}\int_\gw (C_1 a)^4 \, dx < \int_\gw u_i^4\, dx + (C_1 a)^4 |\gw|,  
\eeq
and
\beq \bl{nv}
	\begin{split}
	& \int_\gw C_1 (u_i v_i - u_i w_i + J_e u_i)\, dx \\
	\leq &\, \int_\gw \left(2(C_1 u_i)^2 + \frac{1}{8} v_i^2 + \frac{(C_1 u_i)^2}{r} + \frac{1}{4} r w_i^2 + C_1 u_i^2 + C_1J_e^2 \right) dx    \\
	\leq &\, \int_\gw u_i^4 \, dx + |\gw |\left[C_1^2 \left(2 +\frac{1}{r}\right) + C_1\right]^2 + \int_\gw \left( \frac{1}{8} v_i^2 + \frac{1}{4} r w_i^2 + C_1J_e^2 \right) dx,
	\end{split}
\eeq
and for the memristive coupling term, by completing square of the quadratic form \eqref{vp}, we see that
\beq  \bl{nrh}
	- \int_\gw C_1 k_1 \vp(\rho_i)u_i^2\, dx = - \int_\gw C_1 k_1 (c + \ga \rho_i + \delta \rho_i^2) u_i^2 \leq C_1 k_1 \left( |c| + \frac{\ga^2}{\de} \right) \int_\gw u_i^2\, dx.
\eeq
In \eqref{w1}, 
\beq \bl{nw}
	\int_\gw \frac{1}{r} \, q^2 u_i^2 \, dx \leq \int_\gw \left(\frac{u_i^4}{2} + \frac{q^4}{2r^2}\right) dx  \leq \int_\gw u_i^4\, dx + \frac{q^4}{r^2}\, |\gw|.
\eeq
		
Substitute the above term estimates \eqref{bu4} through \eqref{nw} into the differential inequalities \eqref{u1} and \eqref{w1}. Then sum up the resulted inequalities \eqref{u1}, \eqref{v1}, \eqref{w1} and \eqref{rh1}. One has
\beq \label{g2} 
	\begin{split}
	&\frac{1}{2} \frac{d}{dt} \sum_{i = 1}^m \left(C_1 \|u_i\|^2 + \|v_i \|^2 + \|w_i \|^2) + \|\rho_i \|_{L^4}^4 \right) + \sum_{i = 1}^m \left[C_1 \eta_1 \, \|\nb u_i\|^2 + 3\eta_2\|\rho_i \nb \rho_i\|^2\right] \\
	\leq &\, \sum_{i = 1}^m \int_\gw C_1 (au_i^3 -bu_i^4  + u_i v_i - u_i w_i + J_e  u_i - k_1 \vp (\rho_i) u_i^2)\, dx   \\
	& - \sum_{i = 1}^m \sum_{j = 1}^m\,\int_\gw C_1 P(u_i - u_j)^2\, dx \\
	&+ \sum_{i = 1}^m \int_\gw \left(2\ap^2 +\frac{1}{2} \gb^2 u_i^4 - \frac{3}{8}\, v_i^2 \right) dx + \sum_{i = 1}^m \int_\gw \left(\frac{q^2}{r} (u_1^2 + u_e^2) - \frac{1}{2}\, r w_i^2 \right) dx \\
	&+ \sum_{i = 1}^m \int_\gw \left(\frac{1}{2 k^3_2} u^4_i - \frac{k_2}{2} \rho_i^4 \right) dx   \\
	\leq & \int_\gw (3 - 4) \left(\sum_{i = 1}^m u_i^4\right) dx + C_1 k_1 \int_\gw \left( |c| + \frac{\ga^2}{\de} \right) \int_\gw \left(\sum_{i = 1}^m u^2_i \right) dx \\
	&+ \int_\gw \left(\frac{1}{8} - \frac{3}{8}\right) \left(\sum_{i = 1}^m v^2_i \right)  dx + \int_\gw \left(\frac{1}{4} - \frac{1}{2} \right) \left(\sum_{i = 1}^m r w^2_i \right) dx - \int_\gw \frac{k_2}{2}\, \left(\sum_{i = 1}^m \rho^4_i \right) dx  \\
	&+ m |\gw | \left( (C_1 a)^4 + C_1 J_e^2  + \left[C_1^2 \left(2 +\frac{1}{r}\right) + C_1\right]^2 + 2 \ap^2 + \frac{q^2 u_e^2}{r} + \frac{q^4}{r^2} \right) \\
	= &\, - \int_\gw \sum_{i = 1}^m \left( u_i^4 + \frac{1}{4} v_i^2 + \frac{1}{4} r\, w_i^2 + \frac{1}{2} k_2\, \rho_i^4 \right) dx + C_1 k_1 \int_\gw \left( |c| + \frac{\ga^2}{\de} \right) \int_\gw \left(\sum_{i = 1}^m u^2_i \right) dx  \\
	&+ m |\gw | \left((C_1 a)^4 + C_1 J_e^2  + \left[C_1^2 \left(2 +\frac{1}{r}\right) + C_1\right]^2 + 2\ap^2 + \frac{q^2 u_e^2}{r} + \frac{q^4}{r^2} \right). 	
	\end{split}
\eeq 
Note that by square completion there is a fixed positive constant 
$$
	C_2 = \frac{1}{4} \left[C_1 k_1 \left( |c| + \frac{\ga^2}{\de} \right) + \frac{1}{4} \right]^2
$$ 
such that 
$$ 
	- u_i^4 + C_1 k_1 \left( |c| + \frac{\ga^2}{\de} \right) u^2_i \leq - \frac{1}{4} C_1 u^2_i + C_2,   \quad \text{for} \; 1 \leq i \leq m.
$$
Then from \eqref{g2} we obtain the differential inequality 
\begin{equation} \bl{GE}
	\begin{split}
        \frac{d}{dt} \sum_{i = 1}^m &\, \left(C_1 \|u_i (t)\|^2 + \|v_i (t)\|^2 + \|w_i (t) \|^2) + \|\rho_i (t)\|_{L^4}^4 \right) + \sum_{i = 1}^m 2 C_1 \eta \, \|\nb u_i (t) \|^2 \\
	\leq &\, - \sum_{i = 1}^m \int_\gw \left(\frac{1}{2} C_1 u_i^2 + \frac{1}{2} v_i^2 + \frac{1}{2} r\, w_i^2+ k_2 \,\rho_i^4 \right) dx + M |\gw| \\
	\leq &\, - \frac{1}{2} \sum_{i = 1}^m \int_\gw \left(C_1 u_i^2 + v_i^2 + r\, w_i^2+ k_2 \,\rho_i^4\right) dx + M |\gw|,
	\end{split}
\end{equation}
where
\beq \bl{M} 
	M =  m \left[C_2 + (C_1 a)^4 + C_1 J_e^2  + \left[C_1^2 \left(2 +\frac{1}{r}\right) + C_1\right]^2 + 2 \ap^2 + \frac{q^2 u_e^2}{r} + \frac{q^4}{r^2} \right] . 
\eeq
Set $\gl = \frac{1}{2} \min \{1, r, k_2\}$. From \eqref{GE} it follows that
\begin{equation} \label{ge}
	\begin{split}
	\frac{d}{dt} \sum_{i = 1}^m &\, \left(C_1 (\|u_i\|^2 + \|v_i\|^2 + \|w_i\|^2 + \|\rho_i\|_{L^4}^4\ \right) + \sum_{i = 1}^m 2 C_1 \eta \, \|\nb u_i (t) \|^2  \\
	&+ \gl \, \sum_{i = 1}^m \left( C_1 \| u_i\|^2 + \| v_i \|^2 + \|w_i\|^2 + \|\rho_i\|_{L^4}^4 \right) \leq M |\gw |,
	\end{split}
\end{equation}
for $t \in I_{max} = [0, T_{max})$, which is the maximal time interval of solution existence.
	
We can apply Gronwall inequality to \eqref{ge} (gradient terms removed) to reach the following estimate for all weak solutions $g(t; g^0)$ with $g^0 \in \Psi$ of the problem \eqref{pb},
\beq \label{dse}
	\begin{split}
	&\interleave g(t; g^0) \interleave = \sum_{i = 1}^m \left(\|u_i (t)\|^2 + \| v_i (t) \|^2 + \|w_i (t)\|^2 + \|\rho_i (t) \|_{L^4}^4 \right)  \\
         \leq &\,\frac{\max \{C_1, 1\}}{\min \{C_1, 1\}}e^{- \gl t} \sum_{i = 1}^m \left( \|u_i^0 \|^2 + \|v_i^0\|^2 + \|w_i^0\|^2 + \|\rho_i^0\|^4_{L^4} \right)  + \frac{M}{\gl\, \min \{C_1, 1\}} |\gw |  \\
         = &\, \frac{\max \{C_1, 1\}}{\min \{C_1, 1\}}\, e^{- \gl t} \interleave g^0 \interleave + \, \frac{M}{\gl \, \min \{C_1, 1\}} |\gw|, \quad \text{for} \;\, t \in [0, \infty).
	\end{split}
\eeq 
Here $I_{max} = [0, \infty)$ because \eqref{dse} shows that every weak solution $g(t; g^0)$ of the initial value problem \eqref{pb} will never blow up in the space $\Psi$ at any finite time for any initial state $g^0 \in \Psi$. Moreover, due to $\|\rho_i(t)\|^2 \leq \|\rho_i (t)\|^2_{L^4} |\gw |^{1/2}$ and \eqref{ss} as well as $\Pi \subset \Psi \subset E$, we can confirm that all the weak solutions with any initial state $g^0 \in E$ for the memristive and diffusive Hindmarsh-Rose equations \eqref{cHR} exist in the state space $E$ globally in time. The proof is completed.
\end{proof}

The global existence and uniqueness of the weak solutions and their continuous dependence on the initial data shown in Theorem \ref{T1} enable us to define the solution semiflow \cite{SY} of the memristive and diffusive Hindmarsh-Rose equations \eqref{cHR} on the space $E$ as follows: 
$$
	S(t): g_0 \longmapsto g(t; g_0) = \text{col}\, (u_i (t, \cdot), v_i (t, \cdot), w_i (t, \cdot), \rho_i (t, \cdot): 1 \leq i \leq m), \quad t \geq 0.
$$
We call this semiflow $\{S(t)\}_{t \geq 0}$ the \emph{memristive Hindmarsh-Rose neural network semiflow} generated by the evolutionary equation \eqref{pb} of the neural network model \eqref{cHR}. 

The next theorem shows that the memristive Hindmarsh-Rose neural network semiflow $\{S(t)\}_{t \geq 0}$ is a dissipative dynamical system in the state space $E$.

\begin{theorem} \label{T2}
	There exists a bounded absorbing set for the memristive Hindmarsh-Rose neural network semiflow $\{S(t)\}_{t \geq 0}$ in the space $E$, which is a bounded ball 
\beq \label{abs}
	B^* = \{ h \in E: \| h \|^2_E \leq K\}
\eeq 
where 
\beq \bl{K}
	K = \left[\frac{M}{\gl \min \{C_1, 1\}} + \sqrt{\frac{M}{\gl \min \{C_1, 1\}}}\,\right] |\gw| + 1.
\eeq
\end{theorem}

\begin{proof} 
From the uniform estimate result shown in the first inequality of \eqref{dse} and $\|\rho_i(t)\|^2 \leq \|\rho_i (t)\|^2_{L^4} |\gw |^{1/2}$ we can assert that 
\beq \label{lsp}
	\limsup_{t \to \infty} \, \|g(t; g^0)\|^2_E < K
\eeq
for all weak solutions of \eqref{pb} and any $g^0 \in E$. Moreover, for any positive constant $\Gamma$ and the bounded set $B = \{h \in E: \|h \|^2_E \leq \Gamma \}$ in $E$, there exists a finite time 
$$
	T_B = \frac{1}{\gl} \log^+ \left(\Gamma\, \frac{\max \{C_1, 1\}}{\min \{C_1, 1\}}\right)
$$
such that $\|g(t; g^0) \|^2_E < K$ for all $t > T_B$ and any $g^0 \in B$. By Definition \ref{Dabsb}, the bounded ball $B^*$ in \eqref{abs} is an absorbing set for this memristive Hindmarsh-Rose neural network semiflow $\{S(t)\}_{t \geq 0}$, which is dissipative in the state space $E$. 
\end{proof}

We can further prove the dissipative dynamics of the semiflow $S(t)_{t \geq 0}$ in the higher-order integrable state space $G = [L^4(\gw) \times L^2 (\gw, \mathbb{R}^2) \times L^4(\gw)]^m$.

\begin{theorem} \bl{Tu}
There exists a constant $D > 0$ independent of any initial state, such that the $u_i$-components of the memristive Hindmarsh-Rose neural network semiflow $\{S(t)\}_{t \geq 0}$ has the uniform dissipative property that for any given bounded set $B \subset G$ there is a finite time $\mathcal{T}_B > 1$ and 

\beq \bl{Ku}
	\sup_{g^0 \in B} \left(\sum_{i = 1}^m \, \int_\gw u_i^4 (t, x) \, dx\right) \leq D, \quad \text{for} \;\;  t > \mathcal{T}_B.
\eeq
Consequently, there exists a bounded absorbing set $\widehat{B}\subset G$ for the semiflow $\{S(t)\}_{t \geq 0}$.
\end{theorem}

\begin{proof}
Sum up the $L^2$ inner-products of the $u_i$-equation with $u_i^3(t, x)$ and use Young's inequality \eqref{Yg} appropriately to treat the integral terms. We can get
\begin{equation}  \bl{uvq}
	\begin{split} 
	&\frac{1}{4} \frac{d}{dt} \sum_{i = 1}^m \|u_i(t)\|^{4}_{L^{4}} + 3 \eta \sum_{i = 1}^m \,\|u_i (t) \nb u_i (t)\|^2    \\
	= &\, \sum_{i = 1}^m \int_\gw \left(au_i^{5} - bu_i^{6} + u_i^{3} (v_i - w_i  +J_e) - k_1 (c + \ga \rho_i + \de \rho_i^2) u_i^4 \right) dx   \\
	&\, - \sum_{i = 1}^m\, \sum_{j = 1}^m P (u_i - u_j)^2 (u_i^2 + u_i u_j + u_j^2)  \\
	\leq &\,\sum_{i = 1}^m\, \int_\gw \left[ \left(C_3 + \frac{1}{4} bu_i^6\right) - b u_i^{6} + \left(\frac{1}{4} bu_i^6 + C_4 \left(v_i^{2} + w_i^{2} + J_e^2\right) \right) + k_1 \left(| c| + \frac{\ga^2}{\de}\right) u_i^4 \right] dx \\
	\leq &\,\sum_{i = 1}^m\, \int_\gw \left[ \left(C_3 + \frac{1}{4} bu_i^6\right) - b u_i^{6} + \left(\frac{1}{4} bu_i^6 + C_4 \left(v_i^{2} + w_i^{2} + J_e^2\right) \right) \right] dx \\
	 &\, + \sum_{i = 1}^m\, \int_\gw \left[\frac{1}{4} \,b u_i^6 + \frac{16 \,k_1^3}{b^2}  \left(| c| + \frac{\ga^2}{\de}\right)^3 \right] dx   \\
	 \leq &\, - \frac{1}{4} \sum_{i = 1}^m \int_\gw b u_i^6\, dx + C_4 \sum_{i = 1}^m \int_\gw (v_i^2 + w_i^2) dx  + m |\gw | \left[C_3 + C_4 J_e^2 + \frac{16\, k_1^3}{b^2} \left[ | c| + \frac{\ga^2}{\de}\right]^3 \right]
	\end{split}
\end{equation} 
where $C_3 (a, b)$ and $C_4 (b)$ are positive constants depending on $a, b$ and only on $b$, respectively. By the absorbing property shown in \eqref{dse}, for any given bounded set $B \subset G$, there is a finite time $T_B > 0$ such that
$$
	C_4 \sum_{i = 1}^m\, \int_\gw (v_i^2(t, x) + w_i^2 (t, x)) \, dx \leq C_4 K, \quad  \text{for} \;\,  t > T_B.
$$
Here $K$ is in \eqref{K}. Since $u^6 + 1 \geq u^4$, from \eqref{uvq} the above inequality implies that 
\beq \bl{u4}
	\frac{d}{dt}\sum_{i = 1}^m \|u_i(t)\|^{4}_{L^4} \, + \,b \sum_{i = 1}^m \int_\gw u_i^4\, dx \leq m |\gw | \left[b + C_3 + C_4 (K + J_e^2) + \frac{16\, k_1^3}{b^2}  \left[ | c| + \frac{\ga^2}{\de}\right]^3 \right].
\eeq
Apply the Gronwall inequality to \eqref{u4} and it yields that for $t \geq t_0 > T_B$,
\beq \bl{ub}
	\sum_{i = 1}^m \|u_i (t)\|_{L^4}^4 \leq e^{-bt} \sum_{i = 1}^m \|u_i (t_0)\|_{L^4}^4 + \frac{m}{b} | \gw | \left[b + C_3 + C_4 (K + J_e^2) + \frac{16\, k_1^3}{b^2}  \left[ | c| + \frac{\ga^2}{\de}\right]^3 \right]. 
\eeq

It remains to bound the $L^4$ norm of the initial state $u(t_0)$.  By Proposition \ref{pps}, for any weak solution of the memristive Hindmarsh-Rose evolutionary equation \eqref{pb}, the $u_i$-components have the regularity
$$
	u(t, \cdot) \in H^1 (\gw) \subset L^4 (\gw), \quad \text{for} \;\; t  > 0.
$$
One can integrate \eqref{ge} over the time interval $[T_B, T_B + 1]$ to get
\begin{equation} \bl{H1b}
	\begin{split}
	 &\sum_{i = 1}^m\, \int_{T_B}^{T_B + 1} C_1 \|u_i (s)\|^2_{H^1}\, ds   \\
	 \leq &\, \frac{1}{\min \{\eta, \gl \}} \sum_{i = 1}^m\, \int_{T_B}^{T_B + 1} C_1 (\eta \|\nb u_i (s)\|^2 + \gl \|u_i (s)\|^2)\, ds    \\[2pt]
	 \leq &\, \frac{1}{\min \{\eta, \gl \}} \max \{C_1, 1\} K +  M |\gw|.
	 \end{split}
\eeq
Hence for any given bounded set $B \subset G$ and any initial state $g^0 \in B$, there exists a time point $t_0 \in (T_B, T_B +1)$ such that 
\beq \bl{ut0}
	 \sum_{i = 1}^m\, \|u_i (t_0)\|^2_{L^4} \leq  \sum_{i = 1}^m\, \widehat{C} \|u_i (t_0)\|_{H^1}^2 \leq \frac{\widehat{C}}{C_1 \min \{\eta, \gl \}} \left(\max \{C_1, 1\}K + M |\gw| \right) 
\eeq
where $\widehat{C}$ is the embedding coefficient of $H^1 (\gw)$ into $L^4 (\gw)$.

Finally, combining the inequalities \eqref{ub} and \eqref{ut0}, we conclude that for any given bounded set $B \subset G$ and any initial state $g^0 \in B$, there exists a finite time 
$\mathcal{T}_B > T_B + 1$ such that the claimed inequality \eqref{Ku} is valid with the uniform bound
\beq \bl{G}
	\begin{split}
	D &= \left[\frac{\widehat{C}}{C_1 \min \{\eta, \gl \}} \left(\,\max \{C_1, 1\}K + M |\gw| \,\right) \right]^2   \\
	&\, + \frac{m}{b} | \gw | \left[b + C_{a,b} + C_4 (K + J_e^2) + \frac{16\, k_1^3}{b^2}  \left(| c| + \frac{\ga^2}{\de}\right)^3 \right].
	\end{split}
\eeq
Put together \eqref{Ku} and Theorem \ref{T2}. It is shown that there exists an absorbing set $\widehat{B}$ for the semiflow $S(t)_{t \geq 0}$ in the space $G$,
$$
	\widehat{B} = \{h = (h_u, h_v, h_w, h_\rho): |h_u|^4 + |h_v|^2 + |h_w|^2 + |h_\rho|^4 \leq K + D\}
$$
where $K$ is given in \eqref{K}. The proof is completed.
\end{proof}

\section{\textbf{Exponential Synchronization of the Neural Network}} 

\begin{definition}
For a model evolutionary equation of a neural network called NW such as \eqref{pb} formulated from the memristive and diffusive Hindmarsh-Rose equations \eqref{cHR}, we define the asynchronous degree of this neural network in a state space (as a Banach space) $Z$ to be
$$
	deg_s \,(\text{NW})= \sum_{1\, \leq i \,< j\, \leq \,m} \left\{ \sup_{g_i^0, \, g_j^0\,  \in \, Z} \, \left\{\limsup_{t \to \infty} \, \|g_i (t; g^0_i) - g_j (t; g^0_j)\|_Z \right\}\right\}
$$ 
where $g_i (t)$ and $g_j (t)$ are any two solutions of \eqref{cHR} with the initial states $g_i^0$ and $g_j^0$ for two neurons $\mathcal{N}_i$ and $\mathcal{N}_j$ in the network, respectively. The neural network is said to be asymptotically synchronized if 
$$
	deg_s \,(\text{NW}) = 0.
$$
If the asymptotic convergence to zero of the difference norm for any two neurons in a network admits a uniform exponential rate, then the neural network is called exponentially synchronized. 
\end{definition}

The following exponential synchronization theorem is the main result of this paper.

\begin{theorem} \bl{ES}
Exponential synchronization in the state space $E$ occurs for the memristive Hindmarsh-Rose neural network semiflow $\{S(t)\}_{t \geq 0}$ generated by the weak solutions of the evolutionary equation \eqref{pb}, if the network neuron-coupling strength coefficients $P$ and $Q$ satisfy the threshold conditions 
\beq \bl{Pcn}
	P >  \frac{1}{m} \left[\frac{4 a^2}{b} + \frac{8\beta^2}{b} \left(1 + \frac{1}{r} \right) + \frac{b}{16 \beta^2} \left( 1 + \frac{q^2}{r}\right) + k_1 \left(| c | + \frac{\ga^2}{4\de} \right) \right]
\eeq
and
\beq \bl{Qcn}
	Q \geq  \frac{1}{2m} \left[\left(1 + \frac{32 \beta^2 k_1^2 \ga^2}{b^2}\right) + \frac{K^2}{2\eta_2^3}\left(\frac{64 \beta^2 C^* k_1^2 \de^2}{b^2}\right)^4\right] ,
\eeq
where $K$ in \eqref{K} and $C^*(\gw)$ in \eqref{GN} are uniform constants independent of any initial states.
\end{theorem}

\begin{proof} 
Let the solutions of \eqref{cHR} in the space $L^2 (\gw, \mathbb{R}^4)$ for any two single neurons $\mathcal{N}_i$ and $\mathcal{N}_j$ in this neural network $\mathcal{NW}$ be denoted by 
\begin{equation*}
	\begin{split}
	g_i (t) = \text{col}\, (u_i (t), v_i (t), w_i (t), \rho_i (t)),   \quad   g_j (t) &= \text{col}\, (u_j (t), v_j (t), w_j (t), \rho_j (t)
	\end{split}
\end{equation*}  
with the initial states $g_i^0 = \text{col} \,(u_i^0, v_i^0, w_i^0, \rho_i^0)$ and $g_j^0 = \text{col} \,(u_j^0, v_j^0, w_j^0, \rho_j^0)$ respectively. Denote by $U(t) = u_i (t) - u_j (t),  V(t) = v_i (t) - v_j (t),  W(t) = w_i( t) - w_j (t),  R (t) = \rho_i (t) - \rho_j (t)$. Then $g_i (t) - g_j (t) = \text{col}\, (U(t), \,V(t), \,W(t), \,R(t))$.

Subtraction of the corresponding pairs of component equations for neurons $\mathcal{N}_i$ and $\mathcal{N}_j$ in the model system \eqref{cHR} gives us the following equations for $g_i (t) - g_j (t)$:
\beq \bl{dHR} 
	\begin{split}
		\frac{\pdr U}{\pdr t} & = \eta_1 \,\gd U +  a(u_i + u_j)\,U- b(u_i^2 + u_i u_j + u_j^2)\,U + V - W   \\
		& - k_1 (c + \ga \rho_i + \delta \rho_i^2)U - k_1 (\ga R + \de R(\rho_i + \rho_j)) u_j - m PU   \\
		& = \eta_1 \,\gd U +  a(u_i + u_j)\,U- b(u_i^2 + u_i u_j + u_j^2)\,U + V - W   \\
		& - k_1 (c + \ga \rho_j + \delta \rho_j^2)U - k_1 (\ga R + \de R(\rho_i + \rho_j)) u_i - m PU,  \\
		\frac{\pdr V}{\pdr t} & = - \beta (u_j +  u_k)\,U - V,    \\
		\frac{\pdr W}{\pdr t} & = q U - r W,   \\
		\frac{\pdr R}{\pdr t} & = \eta_2\,\gd R + U - k_2 R - m QR.
	\end{split}
\eeq
Here two decompositions of the difference of memristor coupling terms for $\mathcal{N}_i$ and $\mathcal{N}_j$ in the $U$-equation of \eqref{dHR} are equal. We want to conduct \emph{a priori} estimates of the weak solutions of the system \eqref{dHR}.  First take the $L^2$ inner-product of the $U$-equation with $CU(t)$, where the multiplier constant $C > 0$ will be chosen later. It gives us
\beq \bl{Uq} 
	\begin{split}
	&\frac{1}{2} \frac{d}{dt} (C \|U(t)\|^2) + C\eta_1 \|\nb U(t)\|^2 + C mP \|U(t)\|^2 \\[1pt]
	\leq &\, C \int_\gw \left(a\,(u_i + u_j)U^2 - b\,(u_i^2 + u_i u_j + u_j^2) U^2 + (V - W)U\right) dx \\
	- &\, Ck_1 \int_\gw \left(c + \frac{1}{2} \ga (\rho_i + \rho_j) + \frac{1}{2} \de \left(\rho_i^2 + \rho_j^2 \right)\right)U^2\, dx   \\
	- &\, Ck_1 \int_\gw \frac{1}{2} \left(\ga R + \de R(\rho_i + \rho_j) \right) (u_i + u_j) U\, dx.
	\end{split} 
\eeq
We treat the integral terms on the right-hand side of \eqref{Uq} as follow. By Young's inequality \eqref{Yg}, we have 
\beq \bl{vwU}
	\begin{split}
		&\int_\gw \left(C\, a (u_i + u_j)U^2 - C b\,(u_i^2 + u_i u_j + u_j^2) U^2 + C (V - W)U \right) dx \\
		\leq &\, \int_\gw \left[C a (u_i + u_j)U^2 - \frac{C b}{2}(u_i^2 + u_j^2) U^2 + C^2 \left(1 + \frac{1}{r} \right)U^2 + \frac{1}{4} V^2 + \frac{r}{4} W^2 \right] dx \\
		= &\, \int_\gw \left[C a (u_i + u_j)U^2 - \frac{C b}{2}(u_i^2 + u_j^2) U^2 \right] dx   \\
		&+ C^2 \left(1 + \frac{1}{r} \right) \|U(t)\|^2 + \frac{1}{4} \|V(t)\|^2 + \frac{r}{4} \|W(t)\|^2.
	\end{split}
\eeq
Also we have
\beq \bl{rU}
	\begin{split}
	& -Ck_1 \int_\gw \left(c + \frac{1}{2} \ga (\rho_i + \rho_j) + \frac{1}{2} \de (\rho_i^2 + \rho_j^2))\right)U^2\, dx   \\[3pt]
	\leq &\, Ck_1 \int_\gw \left[\left(| c | + \frac{\ga^2}{4\de}\right) U^2 + \frac{\de}{4}\, (\rho_i + \rho_j)^2 U^2 - \frac{\de}{2}\, (\rho_i^2 + \rho_j^2) U^2 \right] dx   \\[3pt]
	\leq &\, Ck_1 \left(| c | + \frac{\ga^2}{4\de}\right) \|U(t)\|^2
	\end{split}
\eeq
and by using the Young's inequality \eqref{Yg}, 
\beq \bl{RU}
	\begin{split} 
	& - Ck_1 \int_\gw \frac{1}{2} \left(\ga R + \de R(\rho_i + \rho_j) \right) (u_i + u_j) U\, dx \\[3pt]
	\leq &\, \frac{2Ck_1^2 \ga^2}{b} \|R(t)\|^2 + \int_\gw \frac{Cb}{16} (u_i^2 + u_j^2) U^2\,dx \\[3pt]
	+ &\, \frac{4Ck_1^2 \de^2}{b} \int_\gw R^2 (\rho_i^2+ \rho_j^2)\, dx + \int_\gw \frac{Cb}{16}(u_i^2 + u_j^2) U^2\, dx.
	\end{split} 
\eeq

Next for the other three component equations in \eqref{dHR}, we take the $L^2$ inner-products of the $V$-equation with $V(t)$, the  $W$-equation with $W(t)$, and the $R$-equation with $R(t)$ respectively. Then sum them up to get
\beq \bl{eG}
	\begin{split}
	&\frac{1}{2} \frac{d}{dt} \left(\|V \|^2 + \|W \|^2 + \|R \|^2 \right) + \|V \|^2 + r\, \|W \|^2 + \eta_2 \|\nb R\|^2 + (k_2 + mQ)\|R \|^2    \\[2pt]
	\leq &\, \int_\gw \left( - \beta (u_j +  u_k)UV + qUW + UR \right) dx \\ 
	\leq &\, \int_\gw 2 \beta^2 (u_1^2 +  u_2^2)U^2 dx + \frac{q^2}{2r}\|U\|^2 + \frac{1}{4} \|V\|^2 + \frac{r}{2} \|W\|^2 + \frac{1}{2} \left(\|U \|^2 + \|R\|^2\right).
	\end{split} 
\eeq

Now we choose the constant multiplier $C > 0$ to be 
\beq \bl{C}
	C = \frac{8 \beta^2}{b}.
\eeq
Then we see from \eqref{vwU}, \eqref{RU} and \eqref{eG} that 
$$
	\int_\gw \left(- \frac{Cb}{2} + \frac{Cb}{8} + 2\beta^2 \right) (u_i^2 + u_j^2)\, U^2\, dx = - \int_\gw \frac{Cb}{8} (u_i^2 + u_j^2)\, U^2\, dx.
$$
It will be used in the second step of the next differential inequality.

Finally assemble together the two major differential inequalities \eqref{Uq} amd \eqref{eG} with substitution of the term estimates \eqref{vwU}, \eqref{rU} and \eqref{RU}. We obtain 
\beq \bl{Mq} 
	\begin{split}
	&\frac{1}{2} \frac{d}{dt} \left(C \|U(t)\|^2 + \|V(t)\|^2 + \|W(t)\|^2 + \|R(t)\|^2 \right) + C\eta_1 \|\nb U\|^2 + \eta_2 \| \nb R\|^2   \\[5pt]
	&\, + C mP \|U(t)\|^2 + \|V(t)\|^2 + r\, \|W(t)\|^2 + (k_2 + mQ) \|R(t)\|^2   \\[4pt]
	\leq &\, \int_\gw \left[C a (u_i + u_j)U^2 - \frac{C b}{2}(u_i^2 + u_j^2) U^2 \right] dx   \\
	&+ C^2 \left(1 + \frac{1}{r} \right) \|U\|^2 + \frac{1}{4} \|V \|^2 + \frac{r}{4} \|W \|^2  + \int_\gw 2 \beta^2 (u_1^2 +  u_2^2)U^2 \, dx \\
	& + \frac{q^2}{2r}\|U\|^2 + \frac{1}{4} \|V\|^2 + \frac{r}{2} \|W\|^2 +  \frac{1}{2} (\|U(t) \|^2 + \|R(t)\|^2) + Ck_1 \left(| c | + \frac{\ga^2}{4\de}\right) \|U(t)\|^2    \\
	&+ \frac{2Ck_1^2 \ga^2}{b} \|R(t)\|^2 + \int_\gw \frac{Cb}{8} (u_i^2 + u_j^2) U^2\,dx + \frac{4Ck_1^2 \de^2}{b} \int_\gw R^2(t, x) (\rho_i^2 + \rho_j^2)\, dx   \\
	\leq &\, \int_\gw \left[C a (u_i + u_j)U^2 - \frac{C b}{8}(u_i^2 + u_j^2) U^2 \right] dx    \\
	&+ \frac{1}{2} \|V(t) \|^2 + \frac{3r}{4} \|W(t) \|^2  + \left(\frac{1}{2} + \frac{2Ck_1^2 \ga^2}{b}\right) \|R(t)\|^2   \\
	&+ \left[C^2 \left(1 + \frac{1}{r} \right) + \frac{q^2}{2r} + Ck_1 \left(| c | + \frac{\ga^2}{4\de}\right) + \frac{1}{2}\right] \|U(t)\|^2   \\
	&+ \frac{4Ck_1^2 \de^2}{b} \int_\gw R^2(t, x) (\rho_i^2(t, x) + \rho_j^2(t, x))\, dx   \\
	= &\, \int_\gw \left[\frac{4a^2}{b} - \left(\frac{\sqrt{2}a}{b^{1/2}} - \frac{b^{1/2}}{2\sqrt{2}}\, u_i\right)^2 - \left(\frac{\sqrt{2}a}{b^{1/2}} - \frac{b^{1/2}}{2\sqrt{2}}\, u_j \right)^2 \right] CU^2\, dx \\
	&+ \frac{1}{2} \|V(t) \|^2 + \frac{3r}{4} \|W(t) \|^2  + \left(\frac{1}{2} + \frac{2Ck_1^2 \ga^2}{b}\right) \|R(t)\|^2   \\
	&+ \left[C^2 \left(1 + \frac{1}{r} \right) + \frac{q^2}{2r} + Ck_1 \left(| c | + \frac{\ga^2}{4\de}\right) + \frac{1}{2}\right] \|U(t)\|^2   \\
	&+ \frac{4Ck_1^2 \de^2}{b} \int_\gw R^2(t, x) (\rho_i^2(t, x) + \rho_j^2(t, x) )\, dx   \\
	\leq &\, \left[\frac{4C a^2}{b} + C^2 \left(1 + \frac{1}{r} \right) + \frac{q^2}{2r} + Ck_1 \left(| c | + \frac{\ga^2}{4\de}\right) + \frac{1}{2}\right] \|U(t)\|^2   \\ 
	&+ \frac{1}{2} \|V(t) \|^2 + \frac{3r}{4} \|W(t) \|^2  + \left(\frac{1}{2} + \frac{2Ck_1^2 \ga^2}{b}\right) \|R(t)\|^2    \\
	&+ \frac{4Ck_1^2 \de^2}{b} \int_\gw R^2 (t, x) (\rho_i^2 (t, x) + \rho_j^2 (t, x))\, dx, \qquad t > 0. 
	\end{split} 
\eeq
By H\"{o}lder inequality and Theorem \ref{T2}, the integral term at the end of the above inequality \eqref{Mq} satisfies
\beq \bl{Rrho}
	\begin{split}
	&\int_\gw R^2 (t, x) (\rho_i^2 (t, x) + \rho_j^2 (t, x))\, dx   \\
	\leq \|R(t)\|^2_{L^4} &\left(\|\rho_i (t)\|^2_{L^4} + \|\rho_j (t)\|^2_{L^4} \right) \leq \sqrt{2K} \|R(t)\|^2_{L^4}, \quad t > \tau (g^0),
	\end{split}
\eeq
where $\tau (g^0) > 0$ is finite and depends on the initial state $g^0$ in \eqref{pb}, and the constant $K$ is given in \eqref{K}.

Therefore, from \eqref{Mq}, \eqref{Rrho} and Theorem \ref{Tu}, we have shown that there is a finite time $T(g^0) \geq \tau (g^0)$ such that 
\beq \bl{Snc} 
	\begin{split}
	&\frac{1}{2} \frac{d}{dt} \left[C \|U(t)\|^2 + \|V(t)\|^2 + \|W(t)\|^2 + \|R(t)\|^2 \right] + \eta_2 \|\nb R\|^2  \\[6pt]
	&\, + C mP \|U(t)\|^2 + \|V(t)\|^2 + r\, \|W(t)\|^2 + (k_2 + mQ) \|R(t)\|^2   \\[4pt]
	\leq &\, \left[\frac{4C a^2}{b} + C^2 \left(1 + \frac{1}{r} \right) + \frac{q^2}{2r} + Ck_1 \left(| c | + \frac{\ga^2}{4\de}\right) + \frac{1}{2} \right] \|U(t)\|^2   \\ 
	&+ \frac{1}{2} \|V(t) \|^2 + \frac{3r}{4} \|W(t) \|^2  + \left(\frac{1}{2} + \frac{2Ck_1^2 \ga^2}{b}\right) \|R(t)\|^2     \\
	&+ \frac{4Ck_1^2 \de^2}{b} (2K)^{1/2} \|R(t)\|^2_{L^4},   \qquad t > T(g^0). 
	\end{split} 
\eeq
It then yields that
\beq \bl{Syn} 
	\begin{split}
	&\frac{d}{dt} \left[C \|U(t)\|^2 + \|V(t)\|^2 + \|W(t)\|^2 + \|R(t)\|^2 \right] + 2 \eta_2 \|\nb R \|^2    \\[2pt]
	&+ 2C \left[mP - \left[\frac{4 a^2}{b} + C \left(1 + \frac{1}{r} \right) + \frac{1}{2C}\left(1 + \frac{q^2}{r}\right) + k_1 \left(| c | + \frac{\ga^2}{4\de}\right) \right]\right] \|U(t)\|^2   \\[2pt]
	&+ \|V(t)\|^2 + \frac{r}{2}\, \|W(t)\|^2 + 2k_2 \|R(t)\|^2 + 2 mQ \|R(t)\|^2    \\
	\leq &\,\left(1 + \frac{4Ck_1^2 \ga^2}{b}\right) \|R(t)\|^2 + \frac{8 Ck_1^2 \de^2}{b} (2K)^{1/2} \|R(t)\|^2_{L^4}, \quad t > T(g^0).
	\end{split} 
\eeq 
We deal with the last term in \eqref{Syn} by using the Gagliardo-Nirenberg interpolation inequalities \cite[Theorem B.3]{SY}. Since
$$
	H^1 (\gw) \subset L^4 (\gw) \subset L^2(\gw),
$$
one has
$$
	\|R(t)\|^2_{L^4} \leq C^* \|\nb R(t)\|^{2\theta} \|R(t)\|^{2(1 - \theta)}
$$
with a constant $C^*(\gw) > 0$ only depending on the domain $\gw$ and the interpolation index $\theta = 3/4$ given by (note that dim $\gw$ = 3)
$$
	- \frac{3}{4} = \theta \left(1 - \frac{3}{2}\right) - (1 - \theta)\, \frac{3}{2}\, .
$$
Hence the above inequality demonstrates  
\beq \bl{GN}
	\|R(t)\|^2_{L^4} \leq C^* \|\nb R(t)\|^{3/2} \|R(t)\|^{1/2} 
\eeq
and, by Young's inequality together with \eqref{C}, 
\beq \bl{L42}
	\begin{split}
	&\frac{8 Ck_1^2 \de^2}{b} K^{1/2} \|R(t)\|^2_{L^4} \leq  \|\nb R(t)\|^{3/2} \left[ \frac{8 CC^* \,k_1^2 \de^2}{b} (2K)^{1/2} \|R(t)\|^{1/2} \right]    \\
	\leq &\, 2\eta_2 \|\nb R(t)\|^{(3/2) \times (4/3)} + \frac{1}{8\eta_2^3}\left[\frac{8 CC^*\,k_1^2 \de^2}{b}\right]^4 (2K)^2 \|R(t)\|^2 \\
	= &\, 2\eta_2 \|\nb R(t)\|^2 + \frac{K^2}{2\eta_2^3}\left[\frac{64 \beta^2\, C^*\,k_1^2 \de^2}{b^2}\right]^4 \|R(t)\|^2. 
	\end{split}
\eeq 

Under the conditions \eqref{Pcn} and \eqref{Qcn} with the constant $C$ selection \eqref{C}, the network coupling coefficients $P > 0$ and $Q > 0$ satisfy 
\beq \bl{Pk}
	 \xi (P)= 2 \left[mP - \left[\frac{4 a^2}{b} + C \left(1 + \frac{1}{r} \right) + \frac{1}{2C}\left(1 + \frac{q^2}{r}\right) + k_1 \left(| c | + \frac{\ga^2}{4\de}\right) \right] \right]  > 0
\eeq
and 
\beq \bl{Qk} 
	2mQ \|R(t)\|^2 - \left[\left(1 + \frac{4Ck_1^2 \ga^2}{b}\right) + \frac{K^2}{2\eta_2^3}\left[\frac{64 \beta^2 C^* k_1^2 \de^2}{b^2}\right]^4\right] \|R(t)\|^2 \geq 0.   \\
\eeq
Substitute \eqref{L42}, \eqref{Pk} and \eqref{Qk} into the differential inequality \eqref{Syn}. We end up with 
\beq. \bl{Grw}
	\begin{split}
	&\frac{d}{dt} (C \|U(t)\|^2 + \|V(t)\|^2 + \|W(t)\|^2 + \|R(t)\|^2)    \\[5pt]
	&+ \,\kappa \left[ C\|U(t)\|^2 + \|V(t)\|^2 + \|W(t)\|^2 + \|R(t)\|^2 \right]   \\[3pt]
	\leq &\, \frac{d}{dt} (C \|U(t)\|^2 + \|V(t)\|^2 + \|W(t)\|^2 + \|R(t)\|^2)     \\[3pt]
	&+ C\xi \|U(t)\|^2 + \|V(t)\|^2 + \frac{r}{2}\, \|W(t)\|^2 + 2k_2 \|R(t)\|^2 \leq 0,  \quad t > T(g^0).
	\end{split}
\eeq 
in which the coefficient $\kappa$ is given by
\beq \bl{rate}
	\kappa (P) = \min \,\left\{ \xi (P), \,1,\, \frac{r}{2}, \,2k_2 \right\}.
\eeq
Gronwall inequality applied to \eqref{Grw} shows the exponential synchronization result: For any initial state $g^0 \in E$ and any two neurons $\mathcal{N}_i$ and $\mathcal{N}_j$ in the memristive Hindmarsh-Rose neural network $\mathcal{NW}$ modeled in \eqref{cHR}, their gap function $g_i(t; g_i^0) - g_j(t; g_j^0)$ converges to zero in the state space $E$ exponentially at a uniform rate $\kappa (P)$ shown by \eqref{rate} and \eqref{Pk}. Namely, for any $1 \leq i <  j \leq m$,
\beq \bl{ESyn} 
	\begin{split}
	\| g_i (t) - g_j (t) \|_E^2 &= \|U(t)\|^2 + \|V(t)\|^2 + \|W(t)\|^2 + \|R(t)\|^2   \\[5pt]
	&\leq e^{- \kappa \,t } \, \frac{\max \{1, \,C\}}{\min \{1, \,C\}} \, \|g_i^0 - g_j^0 \|^2  \\
	&= e^{- \kappa \,t } \, \frac{\max \{1, \,8\beta^2/b\}}{\min \{1, \,8\beta^2/b\}} \, \left\|g_i^0 - g_j^0 \right\|^2  \to 0, \;\; \text{as} \;\, t \to \infty. 
	\end{split}
\eeq
Hence it is proved that
$$
	deg_s (\mathcal{NW}) = \sum_{1 \,\leq \,i  \,<  \,j \,\leq \,m} \left\{\sup_{g^0\, \in \, E} \, \left\{\limsup_{t \to \infty} \|g_i (t) -g_j(t) \|^2_{L^2(\gw, \mathbb{R}^4)} \right\}\right\} = 0.
$$
Thus the exponential synchronization of this memristive and diffusive Hindmarsh-Rose neural network $\mathcal{NW}$ modeled by \eqref{cHR} in the space $E$ is proved under the threshold conditions \eqref{Pcn} and \eqref{Qcn}. The proof is completed.
\end{proof}

\textbf{Conclusions} \, In this paper, a new model of memristive and diffusive Hindmarsh-Rose neural networks is proposed. This new model features a hybrid system of two partial differential equations with two ordinary differential equations as well as the linear coupling terms in the membrane potential equations and the memristor equations for all neurons. 

Dissipative dynamics of the solution semiflow in the basic state space $E$ and the higher-order state space $G$ are shown through sophisticated grouping estimates and integral inequality leverage, especially for the highly nonlinear membrane potential PDE with the memductance coupling. It paves the way to explore synchronization of this memristive Hindmarsh-Rose neural network. 

The main result is Theorem \ref{ES}, which provides the explicit threshold conditions of the network coupling strengths $P$ and $Q$ to ensure an exponential synchronization of the neural network at a uniform convergence rate. The spirit of the entire proof is to tackle and control the higher nonlinearity in the memductance-potential effect by the linear network coupling in the integrable state spaces. Many steps of sharp analysis including Gagliardo-Nirenberg interpolation inequalities are carried out and cohesively grouped. The rigorous mathematical proof methodology in this work can be extended to study more complex neural networks described by hybrid differential equations or other types reaction-diffusion PDE models in a broad scope.

\bibliographystyle{amsplain}

\begin{thebibliography}{99}

\bibitem{Ay}
I.K. Aybar, \emph{Memristor-based oscillatory behavior in the FitzHugh-Nagumo and Hindmarsh-Rose models}, Nonlinear Dynamics, \textbf{103} (2021), 2917-2929.

\bibitem{BKG}
Y. Babacan, F. Kacar and K. Gurkan, \emph{A spiking and bursting neuron circuit based on memristor}, Neurocomputing, \textbf{203} (2016), 86-91.

\bibitem{BB} 
B. Bao, A. Hu, W. Liu, \emph{et al}, \emph{Hidden bursting firings and bifurcation mechanisms in memristive neuron model with threshold electromagnetic induction}, IEEE Trans. Neural Networks and Learning Systems, \textbf{31} (2020), 502-511.

\bibitem{BRS}
R.J. Buters, J. Rinzel and J.C. Smith, {\em Models respiratory rhythm generation in the pre-B\"{o}tzinger complex, I. Bursting pacemaker neurons}, J. Neurophysiology, \textbf{81} (1999), 382--397.

\bibitem{CV}
V.V. Chepyzhov and M.I. Vishik, {\em Attractors for Equations of Mathematical Physics},  AMS Colloquium Publications, Vol. \textbf{49}, AMS, Providence, RI, 2002.

\bibitem{Chua}
L. Chua, \emph{Memristor - the missing circuit element}, IEEE Trans. Circuit Theory, \textbf{18} (1971), 507.

\bibitem{ChuaK}
L. Chua and S.M. Kang, \emph{Memristive devices and systems}, Proceedings of the IEEE, \textbf{64}(2) (1976), 209-223.

\bibitem{CS}
L.N. Cornelisse, W.J. Scheenen, W.J. Koopman, E.W. Roubos and S.C. Gielen, {\em Minimal model for intracellular calcium oscillations and electrical bursting in melanotrope cells of Xenopus Laevis}, Neural Computations, \textbf{13} (2000), 113--137.

\bibitem{ET}
G.B. Ementrout and D.H. Terman, {\em Mathematical Foundations of Neurosciences}, Springer, 2010. 

\bibitem{EE}
A.S. Et\'{e}m\'{e} \emph{et al}, \emph{Chaos break and synchrony enrichment within Hindmarsh-Rose-type memristive neural models}, Nonlinear Dynamics, \textbf{105} (2021), 785-795.

\bibitem{FH}
R. FitzHugh, {\em Impulses and physiological states in theoretical models of nerve membrane}, Biophysical Journal, \textbf{1} (1961), 445--466.

\bibitem{Guan}
W. Guan, S. Yi and Y. Quan, \emph{Exponential synchronization of coupled memristive neural networks via pinning control}, Chinese Physics B, \textbf{22} (2013), 050504.

\bibitem{HR}
J.L. Hindmarsh and R.M. Rose, {\em A model of neuronal bursting using three coupled first-order differential equations}, Proceedings of the Royal Society London, Ser. B: Biological Sciences,  \textbf{221} (1984), 87--102.

\bibitem{HH}
A. Hodgkin and A. Huxley, {\em A quantitative description of membrane current and its application to conduction and excitation in nerve}, J. Physiology, Ser. B,  \textbf{117} (1952), 500--544.

\bibitem{HY}
M. Hui and J. Yan, \emph{Integral sliding mode exponential synchronization of inertial memristive neural networks with time varying delays}, Neural Processing Letters, accepted, (2022). https://doi.org/ 10.1007/s11063-022-10981-9.

\bibitem{IG}
G. Innocenti and R. Genesio, {\em On the dynamics of chaotic spiking-bursting transition in the Hindmarsh-Rose neuron}, Chaos, \textbf{19} (2009), 023124.

\bibitem{EI}
E.M. Izhikecich, {\em Dynamical Systems in Neuroscience: The Geometry of Excitability and Bursting}, MIT Press, Cambridge, Massachusetts, 2007.

\bibitem{MFL}
S.Q. Ma, Z. Feng and Q. Lu, {\em Dynamics and double Hopf bifurcations of the Rose-Hindmarsh model with time delay}, International Journal of Bifurcation and Chaos, \textbf{19} (2009), 3733--3751.

\bibitem{PYS}
C. Phan, Y. You and J. Su, \emph{Global dynamics of partly diffusive Hindmarsh-Rose equations in neurodynamics}, Dynamics of Partial Differential Equations, \textbf{18}(1) (2021), 33-47.

\bibitem{PY}
C. Phan and Y. You, \emph{A new model of coupled Hindmarsh-Rose neurons}, Journal of Nonlinear Modeling and Analysis, \textbf{1}(1) (2020), 1-16.

\bibitem{CPY}
C. Phan and Y. You, \emph{Synchronization of boundary coupled Hindmarsh-Rose neuron network}, Nonlinear Analysis: Real World Applications, \textbf{55} (2020), 103139.

\bibitem{PSY}
C. Phan, L. Skrzypek and Y. You, \emph{Dynamics and synchronization of complex neural networks with boundary coupling}, Analysis and Mathematical Physics, (2022), 12:33. http://doi. org/10.1007/s13324-021-00613-1.

\bibitem{RJ}
K. Rajagopal \emph{et al}, \emph{Effect of magnetic induction on the synchronizability of coupled neuron network},, Chaos, \textbf{31} (2021), 083115.

\bibitem{RM}
B. Ramakrishnan \emph{et al}, \emph{A new memristive neuron map model and its network's dynamics under electrochemical coupling}, Electronics, \textbf{11} (2022), 153. 

\bibitem{QW}
G. Qi and Z. Wang, \emph{Modeling and dynamics of double Hindmarsh-Rose neuron with memristor-based magnetic coupling and time delay}, Chinese Physics B, \textbf{30}(12) (2021), 120516.

\bibitem{SY} 
G.R. Sell and Y. You, {\em Dynamics of Evolutionary Equations}, Applied Mathematical Sciences, Volume \textbf{143}, Springer, New York, 2002.

\bibitem{SR}
P.P. Singh, A. Rai and B.K. Roy, \emph{Memristor-based asymmetric extreme multistate hyperchaotic system with a line of equilibria, coexisting attractors, its implementation and nonlinear active-adaptive projective synchronization}, European Physical Journal Plus, (2022), 137:875. https://doi.org/10.1140/epjp/s13360-022-03063-1.

\bibitem{SS}
D.B. Strukov \emph{et al}, \emph{The missing memristor found}, Nature, \textbf{453} (2008), 80.

\bibitem{Tr}
D. Terman, {\em Chaotic spikes arising from a model of bursting in excitable membrane}, J. Appl. Math., \textbf{51} (1991), 1418--1450.

\bibitem{US2}
K. Usha and P.A. Subha, \emph{Energy feedback and synchronous dynamics of Hindmarsh-Rose neuron model with memristor}, Chinese Physics B, \textbf{28}(2) (2019), 020502.

\bibitem{VK}
C.K. Volos \emph{et al}, \emph{Memristor: A new concept in synchronization of coupled neuromorphic circuits}, J. Eng. Sci. Tech. Review, \textbf{8} (2015), 157.

\bibitem{WP}
X. Wang, Ju H. Park, Z. Liu, and H. Yang, \emph{Dynamic event-triggered control for GSES of memristive neural networks under multiple cyber-attacks}, IEEE Transactions on Neural Networks and Learning Systems, accepted October 24, 2022.

\bibitem{WS}
Z.L. Wang and X.R. Shi, {\em Chaotic bursting lag synchronization of Hindmarsh-Rose system via a single controller}, Applied Mathematics and Computation, \textbf{215} (2009), 1091--1097.

\bibitem{Wu}
F. Wu, H. Gu and Y. Li, \emph{Inhibitory electromagnetic induction current induces enhancement instead of reduction of neural bursting activities}, Communications in Nonlinear Science and Numerical Simulation, \textbf{79} (2019), 104924.

\bibitem{XJ}
Y. Xu \emph{et al}, \emph{Synchronization between neurons coupled by memristor}, Chaos, Solitons and Fractals, \textbf{104} (2017), 435.

\bibitem{Y}
Y. You, \emph{Global dynamics of diffusive Hindmarsh-Rose equations with memristors}, Nonlinear Analysis: Real World Applications, \textbf{71} (2023), 103827.

\bibitem{Su}
F. Zhang, A. Lubbe, Q. Lu and J. Su, {\em On bursting solutions near chaotic regimes in a neuron model}, Discrete and Continuous Dynamical Systems, Ser. S,  \textbf{7} (2014), 1363--1383.
\end{thebibliography}

\end{document}